\documentclass[11pt, a4paper]{amsart}
\usepackage{amsmath}
\usepackage{amssymb}
\usepackage{dsfont}
\usepackage{color}

\newtheorem{prop}{Proposition}[section]
\newtheorem{teo}{Theorem}[section]
\newtheorem{lema}{Lemma}[section]
\newtheorem{coro}{Corollary}[section]

\theoremstyle{remark}
\newtheorem{rem}{Remark}[section]

\def\ep{\varepsilon}

\def\R{\mathbb R}
\def\N{\mathbb N}

\def\X{{\mathcal X}}
\def\H{{\mathcal H}}
\newcommand{\vvvert}{|\hspace*{-1pt}|\hspace*{-1pt}|}

\begin{document}

\title[Nonlocal diffusion
in domains with holes]{Asymptotic behavior for a nonlocal diffusion
equation in domains with holes}

\author[Cort\'{a}zar, Elgueta, Quir\'{o}s \and Wolanski]{C. Cort\'{a}zar, M. Elgueta, F. Quir\'{o}s \and N. Wolanski}

\address{Carmen Cort\'{a}zar\hfill\break\indent
Departamento  de Matem\'{a}tica, Pontificia Universidad Cat\'{o}lica
de Chile \hfill\break\indent Santiago, Chile.} \email{{\tt
ccortaza@mat.puc.cl} }

\address{Manuel Elgueta\hfill\break\indent
Departamento  de Matem\'{a}tica, Pontificia Universidad Cat\'{o}lica
de Chile \hfill\break\indent Santiago, Chile.} \email{{\tt
melgueta@mat.puc.cl} }

\address{Fernando Quir\'{o}s\hfill\break\indent
Departamento  de Matem\'{a}ticas, Universidad Aut\'{o}noma de Madrid
\hfill\break\indent 28049-Madrid, Spain.} \email{{\tt
fernando.quiros@uam.es} }

\address{Noemi Wolanski \hfill\break\indent
Departamento  de Matem\'{a}tica, FCEyN,  UBA,
\hfill\break \indent and
IMAS, CONICET, \hfill\break\indent Ciudad Universitaria, Pab. I,\hfill\break\indent
(1428) Buenos Aires, Argentina.} \email{{\tt wolanski@dm.uba.ar} }

\thanks{All authors supported by  FONDECYT grants 7090027 and 1110074. The third author supported by
the Spanish Project MTM2008-06326-C02-01. The fourth author supported by the Argentine Council of Research, CONICET under the project PIP625, Res. 845/10 and UBACYT X117. N. Wolanski is a member of CONICET}



\date{}

\begin{abstract}
The paper deals with the asymptotic behavior of solutions to a
non-local diffusion equation,  $u_t=J*u-u:=Lu$, in an exterior domain,
$\Omega$, which excludes one or several holes, and with zero
Dirichlet data on $\mathbb{R}^N\setminus\Omega$. When the space
dimension is three or more this behavior is given by a multiple of
the fundamental solution of the heat equation away from the holes.
On the other hand, if the solution is scaled according to its decay
factor, close to the holes it behaves like a function that is
$L$-harmonic, $Lu=0$, in the exterior domain and vanishes in its complement.
The height of such a function  at infinity is determined through a  matching
procedure with the multiple of the fundamental solution of the heat equation
representing the outer behavior. The inner and the outer behavior
can be presented in a unified way through a suitable global
approximation.
\end{abstract}

\maketitle

\date{}

\section{Introduction}
\label{Intro} \setcounter{equation}{0}

Let  $\mathcal{H}\subset \mathbb{R}^N$, $N\ge3$,  be a  bounded open set with
smooth boundary  and let $\Omega=\mathbb{R}^N\setminus \mathcal{H}$. We do not
assume $\mathcal{H}=\Omega^c$  to be connected, so that it may represent one
or several holes in an otherwise homogeneous medium. Our goal is to
study the large-time behavior of the solution to the nonlocal heat
equation  in that exterior domain with zero data on
the boundary, namely,
\begin{equation}\label{problem}
\begin{cases}
\displaystyle \partial_t u(x,t)=\int_{\mathbb{R}^N} J(x-y)\big(u(y,t)-u(x,t)\big)\,dy\quad&\mbox{in}\quad \Omega\times(0,\infty),\\
u(x,t)=0&\mbox{in}\quad \mathcal{H}\times(0,\infty),\\
u(x,0)=u_0(x)&\mbox{in}\quad \R^N,
\end{cases}
\end{equation}
with a kernel $J$ that is
assumed to be a nonnegative continuous function with unit integral.

We will restrict ourselves to the case where $J$ is smooth, radially
symmetric, with a compact support contained in the unit ball
centered at the origin and $J(0)>0$. However, results similar to the ones we obtain here should hold
for more general kernels.

The hypotheses on the initial data, $u_0$, are that they are
nonnegative, bounded, integrable, and identically zero in the
hole $\mathcal{H}$. We also assume, without loss of generality, that $0\in\mathcal{H}$.

Evolution problems with this type of diffusion have been widely
considered in the literature, since they can be used to model the
dispersal of a species by taking into account long-range effects,
 \cite{BZ}, \cite{CF}, \cite{Fi}. It has been also proposed as a model for phase transitions, \cite{BCh1}, \cite{BCh2}, and, quite recently, for image enhancement, \cite{GO}.

The nonlocal  equation in \eqref{problem} has been the subject of many recent works
in the case where the spatial domain is $\R^N$, and  also when it is
a smooth bounded domain and it is complemented by a
suitable Dirichlet or Neumann type boundary condition, see the recent monograph~\cite{AndreuMazonRossiToledoBook2010} and the references therein. However, up to our knowledge, this is the first time that an exterior domain is being considered.

When there are no holes, $\mathcal{H}=\emptyset$, mass is conserved, $\int_{\mathbb{R}^N}u(x,t)\,dx=
\int_{\mathbb{R}^N}u_0(x)\,dx$, and the solution to~\eqref{problem} behaves for large times as the solution $v$ to the local heat equation with diffusivity
\begin{equation}
\label{eq:definition.alpha}
\alpha:=\frac1{2N}\int_{\R^N}|z|^2J(z)\,dz
\end{equation}
and the same initial data
\cite{CCR},~\cite{IR1}.
More precisely,
\begin{equation}
\label{eq:asymptotic.CCR}
\lim_{t\to\infty}t^{N/2}\max_{x\in\mathbb{R}^N}|u(x,t)-v(x,t)|=0.
\end{equation}
Hence the asymptotic behavior of $u$ can be described in terms of
the fundamental solution of the heat equation with diffusivity $\alpha$,
$\Gamma_\alpha$. This  solution has a self-similar structure,
$$
\Gamma_{\alpha}(x,t)=t^{-N/2}U_\alpha\left(\frac{x}{t^{1/2}}\right),\qquad
U_\alpha(y)=(4\pi \alpha)^{-N/2}e^{-\frac{|y|^2}{4\alpha}}.
$$
Therefore, in self-similar variables we have convergence towards the
stationary profile  $MU_\alpha$, where
$M=\int_{\mathbb{R}^N}u_0(x)\,dx$, that is,
\begin{equation}
\label{eq:behavior.whole.space}
\lim_{t\to\infty}\max_{y\in\mathbb{R}^N}|t^{N/2}u(yt^{1/2},t)-M U_\alpha(y)|=0.
\end{equation}
Thus, there is an asymptotic symmetrization: no matter wether the
initial datum is radial or not,   the large time behavior of $u$ is
given by a radial profile, which, of course, has the same mass as
the datum.

In the presence of holes the situation is very different. On one
hand, mass is not conserved. On the other hand, the presence of the
hole breaks (in general) the symmetry of the spatial domain, and an
asymptotic symmetrization is no longer possible. Nevertheless, when $N\ge3$ we
still have a conservation law, $\int_{\mathbb{R}^N}
u(x,t)\phi(x)\,dx=\int_{\mathbb{R}^N} u_0(x)\phi(x)\,dx$, with
$\phi$ the unique solution to
\begin{equation}\label{stationary}
\begin{cases}
J*\phi=\phi\quad&\mbox{in}\quad\Omega,\\
\phi=0&\mbox{in} \quad \mathcal{H},\\
\phi(x)\to 1&\mbox{as }|x|\to\infty.
\end{cases}\end{equation}
Moreover, there is a non-trivial asymptotic mass which
coincides with the conserved quantity $M^*=\int_{\mathbb{R}^N}
u_0(x)\phi(x)\,dx$. Besides, if we stand far away from the holes,
they will be seen as a point, and we may still expect some
symmetrization. Indeed, we will prove that
$$
\lim_{t\to\infty}t^{N/2}u(yt^{1/2},t)=M^*U_\alpha(y)\quad\mbox{uniformly in }
|y|\ge \delta>0.
$$
The only effect of the holes in this \emph{outer} limit is the loss
of mass. However, close to the holes, in the \emph{inner} limit,
their effect  is much more important. On compact sets solutions
still decay as $O(t^{-N/2})$. If we scale the solutions accordingly,
we get that the new variable $w(x,t)=t^{N/2}u(x,t)$ satisfies
$$
\partial_t w=J*w-w+\frac N2\frac wt.
$$
Thus, we expect $w$ to converge to a $L$-harmonic function $\Phi$ with zero \lq boundary data',
$$
L\Phi:=J*\Phi-\Phi=0, \quad x\in\Omega,\qquad \Phi=0, \quad x\in \mathcal{H}.
$$
This problem does not have uniqueness. However, a unique solution
can be  determined by prescribing the value at infinity. Indeed, if
$\Phi\to C^*$ as $|x|\to\infty$, then $\Phi=C^*\phi$.

In order to select the right constant $C^*$ characterizing the asymptotic limit of $w$, we need some extra information. This will come from the outer limit, following a typical matched asymptotics procedure.  Let us explain in some detail how this is done.

In view of the asymptotic behavior of $\phi$, we can describe the outer limit in the alternative form
$$
t^{N/2}|u(x,t)-M^*\phi(x)\Gamma_\alpha(x,t)|\to 0\quad\mbox{as }
t\to\infty \mbox{ uniformly in } |x|\ge\delta t^{1/2},\quad \delta>0.
$$
Besides, since $\lim\limits_{t\to\infty}t^{N/2}\Gamma_\alpha(x,t)=(4\pi\alpha)^{-N/2}$ uniformly on compact sets, the expected inner limit can be written as
$$
t^{N/2}|u(x,t)-C^*(4\pi\alpha)^{N/2}\phi(x)\Gamma_\alpha(x,t)|\to 0\quad\mbox{as }
t\to\infty \text{ uniformly in  compact sets}.
$$
We will prove that we can go beyond compact sets, and that this limit holds in a set of the form $|x|\le\delta t^{1/2}$ for some $\delta>0$. Hence, there is an overlapping region between the inner and the outer development, and they can be matched, leading to
$$
C^*=M^*(4\pi\alpha)^{-N/2}.
$$
Notice that we have been able to describe the inner and the outer behavior in a unified way. Hence we may gather the two results in just one theorem.
\begin{teo}
\label{thm:main}
Let $N\ge3$, $0\le u_0\in L^1(\R^N)\cap L^\infty(\R^N)$ and $M^*=\int_{\mathbb{R}^N} u_0(x)\phi(x)\,dx$. Then,
\begin{equation}\label{result}
t^{N/2}|u(x,t)-M^*\phi(x)\Gamma_\alpha(x,t)|\to 0\quad\mbox{as}\quad
t\to\infty \quad\mbox{uniformly in }\R^N.
\end{equation}
\end{teo}

\begin{rem} The set  $\Omega+B(0;1/2)$ has only one unbounded
component $\mathcal{U}$. This component $\mathcal{U}$ may contain several components of $\Omega$. Though only one of them is really unbounded, due to the non-local character of the diffusion operator, all of them are \lq connected to infinity'. However, there may exist components of $\Omega$ which do not intersect $\mathcal{U}$, and which are, thus, \lq truly' bounded. The function $\phi$ is identically 0 there. Therefore,  the scaling we are using is not adequate to characterize the asymptotic behavior
in such subdomains. Indeed, the decay rate in these kind of bounded
\lq components' is exponential, and the asymptotic profile is an eigenfunction of the operator $L$  with zero Dirichlet boundary conditions in the complement of the \lq component', associated to the first eigenvalue, \cite{CCR}.
\end{rem}

\begin{rem}The asymptotic behavior for the case where there are no holes~\eqref{eq:behavior.whole.space} can also be described in the form~\eqref{result}, since the solution to~\eqref{stationary} when $\mathcal{H}=\emptyset$ is  $\phi=1$.
\end{rem}

When $N\le 2$, mass is expected to decay to zero (logarithmically
when $N=2$,  and like a power when $N=1$), and the  solutions will
decay faster than $O(t^{-N/2})$. The analysis becomes much more
involved, and is postponed to a future paper.

An analogous study in an exterior domain has been performed for the corresponding local
problem  in~\cite{Herraiz}. However, the author makes extensive use
in the proofs of the explicit form of the fundamental solution of
the local heat equation in the whole space. One of the main difficulties in the
non-local case is that we do not have such an explicit expression.

An alternative approach, in which fundamental solutions do not play
such a key role in the proofs, has been used for nonlinear (local
problems), see for instance \cite{BQV}, \cite{RV}. However, the operators under
consideration in these cases are invariant under some scaling
transformations, which is not the case in the present problem.

Nevertheless, although it will not be apparent from our proofs,
scaling still plays a key role in  the {\em outer} region. In fact, as observed in~\cite{CERW2},
under the usual parabolic scaling our operator `converges' to
$\alpha\Delta$. This is the reason why the outer behavior is given by
the fundamental solution to the heat equation with diffusivity $\alpha$.
This idea has been used  in
\cite{TW} to study the asymptotic behavior for  the
nonlocal heat equation with absorption in a certain critical case.

We complete the study of the asymptotic behavior by analyzing the function $\phi$ at infinity. We prove that all the derivatives of $\phi$ behave at infinity as those of the fundamental solution of the Laplace operator.
This result might be of independent interest.

\noindent{\sc Notations. } In what follows we will denote  $Lu(x,t):=\int_{\mathbb{R}^N} J(x-y)\big(u(y,t)-u(x,t)\big)\,dy$.

\section{Preliminaries}\label{Preliminaries}
\setcounter{equation}{0}

\noindent\textsc{On the notion of solution. }
Let $u_0\in L^1(\R^N)$, $u_0=0$ a.e. in $\mathcal{H}$. A solution of~\eqref{problem} is a function $u\in C([0,\infty):L^1(\R^N))$ such that for all $t>0$
    \begin{equation}\label{eq:form.int}
        u(\cdot,t)=u_0+\int_0^t \big(J\ast u(\cdot,s)-u(\cdot,s)\big)\,ds \ \text{ a.e. in }\Omega, \qquad u(\cdot,t)=0\ \text{ a.e. in }\mathcal{H}.
    \end{equation}
Subsolutions and supersolutions are defined as usual by replacing
the equalities in the definition above respectively by $\leq$ or
$\geq$.
\begin{rem}
If $u$ is a solution, then
$u\in L^1(\mathbb{R}^N\times[0,T])$ for all $T>0$.
Hence,~\eqref{problem} holds, not only  in the sense of
distributions, but also~a.e.~in $\Omega\times(0,\infty)$.  Moreover, we also have $u\in C^1([0,\infty): L^1(\R^N))$, and the equation
holds~a.e.~in $\Omega$ for all $t\ge0$.
\end{rem}

\

\noindent\textsc{Existence and uniqueness. } Existence and uniqueness of a solution to \eqref{problem} can be proved  for any $u_0\in L^1(\R^N)$ such that $u_0=0$ a.e.~in $\mathcal{H}$ by using a fixed point argument. We do not need to assume neither $u_0\in L^\infty(\mathbb{R}^N)$, nor a sign restriction.

\begin{teo}
    For any $u_0\in L^1(\R^N)$ such that $u_0=0$ a.e.~in $\mathcal{H}$ there exists a unique solution of~\eqref{problem}.
\end{teo}
\begin{proof}
The proof is quite similar to the one for the case $\mathcal{H}=\emptyset$ performed in~\cite{CCR}.

Let $\mathcal{B}_{t_0}=\{u\in C([0,t_0]:L^1(\mathbb{R}^N)):  u(\cdot, t)=0 \text{ a.e. in }\mathcal{H} \text{ for all }t\in[0,t_0]\}$. Endowed with the norm
    $$
      \vvvert u\vvvert=\max_{0\leq t\leq t_0}\|u(\cdot,t)\|_{ L^1(\mathbb{R}^N)},
    $$
it is a Banach space.
We define the operator
$\mathcal{T}:\mathcal{B}_{t_0}\to\mathcal{B}_{t_0}$ through
\begin{equation}
\label{eq:definition.operator}
  (\mathcal{T} u)(\cdot,t)=\begin{cases}
  \displaystyle u_0(\cdot)+\int_0^t\left( J*u(\cdot,s)-u(\cdot,s)\right)\,ds\quad&\text{a.e. in }\Omega,\\[10pt]
  0\quad&\text{a.e. in }\mathcal{H}.
  \end{cases}
\end{equation}
This operator turns out to be contractive if $t_0$ is small enough. Indeed,
$$
\begin{array}{l}
 \displaystyle\int_{\mathbb{R}^N}
 |\mathcal{T}\varphi-\mathcal{T}\psi|(\cdot,t)
 \\[10pt]
\displaystyle \qquad
 \leq\int_{\Omega}\int_0^t \Big( |J*(\varphi-\psi)(\cdot,s)|
  +
  |\varphi-\psi|(\cdot,s)\Big)\,ds \\[10pt]
\displaystyle  \qquad \leq
\int_{0}^t\Big(\|J\|_{L^1(\mathbb{R}^N)}+1\Big)\|(\varphi-\psi)(\cdot,s)\|_{L^1({\mathbb{R}^N)}}\,ds.
\end{array}
$$
Hence, $\vvvert\mathcal{T}\varphi-\mathcal{T}\psi\vvvert\leq 2t_0\vvvert \varphi-\psi\vvvert$.
Thus, $\mathcal{T}$ is a contraction if $t_0<1/2$.
Existence and uniqueness in the time interval $[0,t_0]$ now follow easily, using Banach's fixed point theorem. Since the length $t_0$ of the  existence and uniqueness time interval does not depend on the initial data, we may iterate the argument to extend the result to all positive times.
\end{proof}

\

\noindent\textsc{Comparison. } Comparison is an immediate
consequence of the following  $T$-contraction property in $L^1$.
Again, we are only assuming that the initial data are integrable.

\begin{teo}\label{thm:contraction}
  Let $u_1$ and $u_2$ be two solutions of~\eqref{problem}
  with initial data $u_{1,0},u_{2,0}\in L^1(\R^N)$. Then, for every $t\ge0$,
  $$
  \int_{\mathbb{R}^N} (u_1-u_2)_+(\cdot,t)\leq   \int_{\mathbb{R}^N} (u_{1,0}-u_{2,0})_+.
  $$
\end{teo}

\begin{proof} We subtract  the equations for $u_{1}$ and $u_{2}$
    and multiply by $\mathds{1}_{\{u_{1}>u_{2}\}}$.
    Since $u_{1}-u_{2}\in C^1([0,\infty); L^1(\mathbb{R}^N))$,
    $$
    \partial_t (u_{1}-u_{2})\mathds{1}_{\{u_{1}>u_{2}\}}=\partial_t (u_{1}-u_{2})_+.
    $$
    On the other hand, just using that  $0\leq\mathds{1}_{\{u_{1}>u_{2}\}}\leq 1$ and $J\ge0$ we get,
    $$ J\ast(u_{1}-u_{2})\mathds{1}_{\{u_{1}>u_{2}\}}
    \leq J\ast(u_{1}-u_{2})_+.
    $$
    Finally, $(u_{1}-u_{2})\mathds{1}_{\{u_{1}>u_{2}\}}=(u_{1}-u_{2})_+$.
    We end up with
    $$
    \partial_t (u_1-u_2)_+\leq
    \begin{cases}
    J\ast(u_{1}-u_2)_+-
    (u_{1}-u_2)_+\qquad&\text{a.e. in }\Omega,\\[10pt]
    0\qquad&\text{a.e. in }\mathcal{H}.
    \end{cases}
    $$
    Integrating in space, and using Fubini's Theorem, we get
    $$
    \partial_t \int_{\mathbb{R}^N}(u_{1}-u_{2})_+(\cdot,t)\le0.
    $$
\end{proof}

\begin{rem}
The same proof applies when there are no holes, $\Omega=\mathbb{R}^N$.
\end{rem}

As a corollary we have, on one hand comparison and, on the other hand an $L^1$-contraction property for solutions,
$$
\|(u_1-u_2)(\cdot,t)\|_{L^1(\mathbb{R}^N)}\le
\|u_{1,0}-u_{2,0}\|_{L^1(\mathbb{R}^N)}.
$$
The latter property implies the continuous dependence of solutions on the initial data, and  a uniform control of the
$L^1$-norm along time,
$$
\|u(\cdot,t)\|_{L^1(\R^N)}\le \|u_0\|_{L^1(\R^N)}.
$$

The proof of Theorem~\ref{thm:contraction} applies under much weaker hypotheses, leading to a more general comparison principle that will be useful in the sequel.

\begin{teo}\label{thm:contraction.general}
Let $\{\Omega(t)\}_{t\ge0}$ be a continuous (with respect to the Hausdorff distance)  family of open smooth domains of $\mathbb{R}^N$. Let $u_1,u_2\in C^1([0,\infty);L^1(\mathbb{R}^N))$ be such that
$$
\partial_t u_1-Lu_1\le 0,\ \partial_t u_2-L u_2\ge0\quad\text{for a.e. }x\in\Omega(t),\ \forall\, t>0,
$$
and $u_1(x,t)\le u_2(x,t)$ for $a.e.\ x\in\mathbb{R}^N\setminus \Omega(t)$, $\forall \,t>0$.
 Then, for every $t\ge0$,
  $$
  \int_{\mathbb{R}^N} (u_1-u_2)_+(\cdot,t)\leq   \int_{\mathbb{R}^N} (u_{1}(x,0)-u_{2}(x,0))_+.
  $$
\end{teo}

\begin{rem}
\label{rem:comparison.general} In order to compare a subsolution and a supersolution in $\Omega(t)$, $t>0$, we only need them to be ordered in $\left(\Omega(t)+B(0;1)\right)\setminus\Omega(t)$, $t>0$, and in $\Omega(0)+B(0;1)$.
\end{rem}

\noindent\textsc{Time decay. } Solutions to
problem~\eqref{problem} satisfy
\begin{equation}\label{eq-exponencial}
u(\cdot,t)=\begin{cases}
\displaystyle e^{-t}u_0+\int_0^te^{-(t-s)}J*u(\cdot,s)\,ds\quad&\mbox{a.e. in }\R^N\setminus\H,\\[10pt]
0\quad&\mbox{a.e. in }\H.
\end{cases}
\end{equation}
Therefore, $u$ has the same spatial regularity as the initial datum
$u_0$, but not more.  In particular, if $u_0$ is not bounded, neither is $u(\cdot,t)$ for any later time.

On the contrary, if the initial datum is bounded, the solution stays bounded as time goes by. Moreover, if the datum belongs to $L^1(\R^N)\cap L^\infty(\R^N)$, the solution decays like $O(t^{-N/2})$. To check this, we first observe that, since $u$ is a solution to \eqref{problem} and $u(\cdot,t)=0$ a.e. in $
\mathcal{H}$ for all  $t\in [0,\infty)$, then
\begin{equation}\label{problem1}
u_t=Lu-\X_\mathcal{H} (J*u)\quad\text{in }\R^N\times(0,\infty).
\end{equation}
Besides, if $u_0\ge0$, we have $u(x,t)\ge0$. Hence, $u$, which belongs to $C([0,\infty);L^1(\mathbb{R}^N)$), is a subsolution to the Cauchy
problem (no holes). Therefore,  it lies below the solution
$u_{\text{C}}$ to the Cauchy problem with the same initial datum.

On the other hand, as we have already mentioned (see
\eqref{eq:behavior.whole.space}),  if $u_0\in L^1(\mathbb{R}^N)\cap
L^\infty(\mathbb{R}^N)$, then $u_{\text{C}}=O(t^{-N/2})$. Thus,
\begin{equation}
\label{decaimiento}
0\le u(x,t)\le C
t^{-N/2}.
\end{equation}

\begin{rem} In order to obtain the asymptotic behavior of
$u_{\text{C}}$, the paper \cite{CCR} requires $u_0,\hat
u_0 \in L^1(\mathbb{R}^N)$. The paper \cite{IR1} dispenses with the integrability assumption on $\hat u_0$, and on the sole hypothesis $u_0\in L^1(\mathbb{R}^N)$ proves that
$$
\lim_{t\to\infty}t^{N/2}\max_{x\in\mathbb{R}^N}|u(x,t)-e^{-t}u_0(x)-v(x,t)|=0,
$$
where $v$ has the same meaning as in~\eqref{eq:asymptotic.CCR}. Hence, the hypotheses $u_0\in L^1(\mathbb{R}^N)\cap L^\infty(\mathbb{R}^N)$ are enough to have~\eqref{eq:behavior.whole.space}.
\end{rem}

\

\noindent\textsc{$L^\infty$-solutions. } Instead of the $L^1$-theory
outlined above, an $L^\infty$-theory for~\eqref{problem} is also possible. Given an initial datum $u_0\in L^\infty(\mathbb{R}^N)$, an $L^\infty$-solution to problem~\eqref{problem} is a function  $u\in
C([0,\infty);L^\infty(\mathbb{R}^N))$ such that~\eqref{eq:form.int} holds for all $t\ge0$.

As in the case of $L^1$-solutions, existence and uniqueness follow
from a fixed-point argument. Indeed,
$\mathcal{B}_{t_0}=\{u\in L^\infty(\mathbb{R}^N\times[0,t_0]):  u(\cdot, t)=0 \text{ a.e. in }\mathcal{H} \text{ for all }t\in[0,t_0]\}$ endowed with the
norm
$$
 \vvvert u\vvvert=\max_{0\leq t\leq t_0}\|u(\cdot,t)\|_{ L^\infty(\mathbb{R}^N)},
$$
is a Banach space, and the operator
$\mathcal{T}:\mathcal{B}_{t_0}\to\mathcal{B}_{t_0}$ defined by~\eqref{eq:definition.operator}
satisfies
$\vvvert\mathcal{T}\varphi-\mathcal{T}\psi\vvvert\leq 2t_0\vvvert \varphi-\psi\vvvert$.
Hence, it is contractive if $t_0<1/2$, which implies the local
existence and uniqueness result.
By iteration, taking as initial datum
$u(\cdot,t_0)\in L^\infty(\mathbb{R}^N)$, we obtain existence  and
uniqueness for $[0,2t_0]$ and therefore for all times.

Besides, $L^\infty$-solutions depend continuously on the initial data. Indeed, let $u_1$ and $u_2$ be  $L^\infty$-solutions with initial data $u_{1,0}$, $u_{2,0}\in L^\infty (\mathbb{R}^N)$.
Since $u_i$ satisfies \eqref{eq-exponencial} (with $u_0=u_{i,0}$), an easy computation shows that for every $t>0$,
$$
\max_{s\in[0,t]}\|u_1(\cdot,s)-u_2(\cdot,s)\|_{L^\infty(\mathbb{R}^N)}\leq  \|u_{1,0}-u_{2,0}\|_{L^\infty(\mathbb{R}^N)}.
$$
In particular, the maximum principle for bounded solutions applies, namely
\[
\|u(\cdot,t)\|_{L^\infty(\R^N)}\le \|u(\cdot,0)\|_{L^\infty(\R^N)}\quad\forall\, t>0.
\]

Essentially the same computation yields  a comparison result for bounded sub- and supersolutions if the hole does not change with time. Let us see this in some detail. Assume $u_{1,0}\le u_{2,0}$ in $\R^N$. Let $u_1$ be a bounded subsolution with initial datum $u_{1,0}$ and $u_2$ a bounded supersolution with initial datum $u_{2,0}$ in $\Omega=\R^N\setminus\H$ with $u_1\le u_2$ in $\H$ for every $t>0$. Let $u=u_1-u_2$ and $u_0=u_{1,0}-u_{2,0}$. Then,
\[
u(x,t)\le
\begin{cases}
\displaystyle
e^{-t}u_0(x)+\int_0^t e^{-(t-s)} J*u(x,s)\,ds\quad&\mbox{if } x\in \Omega,\\
0\quad&\mbox{if }x\in\H.
\end{cases}
\]
Hence, for $x\in\Omega$,
\[
u(x,t)\le \int_0^t e^{-(t-s)} J*u(x,s)\,ds\le \sup_{\R^N\times(0,t)}\hskip-8pt u\ (1-e^{-t}),
\]
with leads to
\[
\sup_{\R^N\times(0,t)}\hskip-8pt u\le \sup_{\R^N\times(0,t)}\hskip-8pt u \ (1-e^{-t}).
\]
Therefore, $\sup_{\R^N\times(0,t)} u\le 0$ for every $t>0$. This means that $u\le0$. Thus, $u_1\le u_2$.

\begin{rem} The comparison result for bounded sub- and supersolutions of the Cauchy problem was proved in~\cite{LW} using a different technique. The result of that paper is in fact more general, since it also applies to
semilinear equations, as long as the nonlinearity is locally Lipschitz continuous. In addition, its proof copes without further ado with the case of holes that change with time.
\end{rem}

\begin{rem}
From now on we will always assume that $0\le u_0\in L^1(\R^N)\cap L^\infty(\R^N)$.
\end{rem}

\

\noindent\textsc{Representation formula. }
By \eqref{problem1}, $u$ can be expressed in terms of the fundamental solution $F=F(x,t)$ to the operator $\partial_t-L$ in the whole space by means of the variations of constants formula. Thus,  for $t\ge t_0$ we have
\begin{equation}
\label{eq:variation.of.constants}
u(x,t)=\int_{\mathbb{R}^N}F(x-y,t-t_0)u(y,t_0)\,dy
-\int_{t_0}^t\int_{\mathbb{R}^N}F(x-y,t-s)\X_\mathcal{H}(y)(J*u(\cdot,s))(y)\,dy\,ds.
\end{equation}
The fundamental solution can be decomposed as
\begin{equation}\label{fund-sol}
F(x,t)=e^{-t}\delta(x)+\omega(x,t),
\end{equation}
where $\delta(x)$ is the Dirac mass at the origin in $\R^N$ and $\omega$ is defined via its Fourier transform as
\begin{equation}
\label{eq:transform.omega}
\hat\omega(\xi,t)=e^{-t}\big(e^{\hat J(\xi)t}-1\big),
\end{equation}
see~\cite{CCR}, from where it is easy to see that $\omega$ is a smooth function and
\[
\int_{\mathbb{R}^N}\omega(x,t)\,dx=\hat\omega(0,t)=1-e^{-t}.
\]
Combining~\eqref{eq:variation.of.constants} and~\eqref{fund-sol}, we obtain
\begin{equation}
\label{eq:representation.formula}
\begin{aligned}
u(x,t)&=e^{-(t-t_0)}u(x,t_0)+\int_{\mathbb{R}^N} \omega(x-y,t-t_0)u(y,t_0)\,dy\\
&-\int_{t_0}^te^{-(t-s)}\X_\mathcal{H}(x)(J*u(\cdot,s))(x)\,ds\\
&-\int_{t_0}^t\int_{\mathbb{R}^N} \omega(x-y,t-s)\X_\mathcal{H}(y)(J*u(\cdot,s))(y)\,dy\,ds.
\end{aligned}
\end{equation}
This formula will be the starting point to obtain both the behavior at infinity of $\phi$ and the outer large time behavior  of $u$.

\

\noindent\textsc{Estimates for $\omega$. } In order to take profit of the representation formula \eqref{eq:representation.formula}, we need good estimates for the regular part, $\omega$, of the fundamental solution.

A first source for estimates comes from the asymptotic convergence
of $\omega$ to the fundamental solution of the local heat equation
with diffusivity $\alpha$, which is proved using Fourier transform
techniques  \cite{IR1}. Indeed,  given a multi-index
$\beta=(\beta_1,\dots,\beta_N)\in\mathbb{N}_0^N$ of order
$|\beta|=s$,
\begin{equation}\label{estima-W}
t^{(N+s)/2}\|D^\beta
\omega(\cdot,t)-D^\beta\Gamma_\alpha(\cdot,t)\|_{L^\infty(\R^N)}\to0\quad\mbox{as
}t\to\infty,
\end{equation}
where $D^\beta
f=\partial^{\beta_1}_{x_1}\cdots\partial^{\beta_N}_{x_N}f$. In
particular,
\begin{equation}\label{decaimiento-W}
|D^\beta \omega(x,t)|\le Ct^{-\frac{N+s}2}\quad\text{if}\quad|\beta|=s.
\end{equation}
These estimates give the right order of time decay. However, they do not take into account the spatial structure of $\omega$, and will not be enough for our purposes. We need to know something about the spatial decay of $\omega$ as $|x|\to\infty$.

A second source of estimates is the expansion
\begin{equation}
\label{eq:series.omega}
\omega(x,t)=e^{-t}\sum_{n=1}^\infty\frac{t^nJ^{*n}(x)}{n!},
\end{equation}
which follows by using the Taylor series of the exponential in~\eqref{eq:transform.omega}.
This expression, recently derived in~\cite{BCF}, was used by the authors to obtain estimates giving the  behavior of $\omega$ as $|x|\to\infty$ for each fixed $t$. However, though they give the right order of spatial decay for each time, they become very poor as $t\to\infty$.

Hence we need a different approach: we will  obtain \lq good' estimates through a comparison argument.
Indeed, as observed in~\cite{TW}, $\omega$ is a solution to
\begin{equation}\label{eq-W}
\begin{cases}
\partial_t\omega(x,t)-L\omega(x,t)=e^{-t}J(x)\quad&\mbox{in}\quad\R^N\times(0,\infty),\\
\omega(x,0)=0\quad&\mbox{in}\quad\R^N,
\end{cases}
\end{equation}
a fact that can be checked, for example, either by differentiating $\hat\omega(\xi,t)$ with respect to $t$ or using that $e^{-t}\delta(x)+\omega(x,t)$ is a solution of $\partial_t F-LF=0$. Good estimates will then follow from the use of appropriate barriers.
\begin{prop}
The function $\omega$ is nonnegative. Moreover, given a multi-index $\beta\in\mathbb{N}_0^N$ of order $|\beta|=s$,
\begin{eqnarray}\label{barriers}
|D^\beta \omega(x,t)|\le C\frac t{|x|^{N+2+s}},
\\[10pt]
\label{integral-estimates}
\int_{\R^N}|D^\beta \omega(x,t)|\,dx\le C t^{-s/2}.
\end{eqnarray}
\end{prop}
The cases $s=0,1$ were already proved in \cite{TW}, through a clever use of the above mentioned comparison principle from~\cite{LW}.  The proof of the general case is completely analogous. The fact that $\omega$ is nonnegative was also proved in \cite{TW} and can also be seen from the series expansion~\eqref{eq:series.omega}.

\begin{rem} Estimates \eqref{barriers} and \eqref{integral-estimates} are in some sense optimal, since they are invariant under the \lq parabolic' scaling $(\tau_k u)(x,t)=k^Nu(kx,k^2t)$. This scaling will play a major part in some of our proofs.
\end{rem}

\section{The stationary problem}
\setcounter{equation}{0}

As we have seen in the introduction, the $L$-harmonic function
$\phi$ solving~\eqref{stationary} plays a key role in the description
of the asymptotic behavior of solutions to~\eqref{problem}. In this
section we prove that problem~\eqref{stationary} is well
posed, and  we obtain some properties of $\phi$ that will be
required later. Though not needed in the sequel, we complete our study of the stationary solution $\phi$ by describing precisely its behavior at infinity.

\

\noindent\emph{Remark on the notion of solution. } In principle we
only need $\phi\in L^1_{\text{loc}}(\mathbb{R}^N)$, so that the
convolution makes sense. But the equation implies that $\phi$ has in
$\Omega$ the same regularity as $J$. Hence $\phi$ is at least
continuous in $\Omega$ and, in our case,  $\phi$ is smooth in $\Omega$.  However, $\phi$ is not continuous across the
boundary of $\Omega$.

\

\subsection{Existence and uniqueness}
To prove existence we will approximate the domain $\Omega$ by a
sequence of bounded domains, $\Omega_n=\Omega\cap B(0;n)$. Let us
start by studying the Dirichlet problem for such bounded domains.

\begin{lema}\label{existence and comparison}
Let $G\subset\mathbb{R}^N$ open and bounded.

\begin{itemize}
\item[(i)] Given $f \in
L^\infty(G^c)$, there exists a solution to  problem
$$
L\phi(x)=0\mbox{ in }\overline G, \qquad \phi=f(x)\mbox{ in
}\overline G^c.
$$
\item[(ii)] Let $\phi_1,\phi_2\in L^\infty(\R^N)\cap C(\overline G)$ such that
$$
L\phi_1\ge L\phi_2\mbox{ in }\overline G,\qquad \phi_1\le\phi_2
\mbox{ in }\overline G^c.
$$
Then $\phi_1\le\phi_2$ in $\R^N$.
\end{itemize}
\end{lema}

\begin{proof} \textsf{Existence. } Let
$ {\mathcal K}= \{\psi\in C(\overline G):
\|\psi\|_{L^\infty(G)}\le\|f\|_{L^\infty(G^c)}\}$. We
define an operator  ${\mathcal T}:{\mathcal K}\to{\mathcal K}$
through the formula
$$
\left({\mathcal T}\psi\right)(x)=\int_G
J(x-y)\psi(y)\,dy+\int_{G^c} J(x-y)f(y)\,dy.
$$
It is easy to see that $\|{\mathcal
T}\psi\|_{L^\infty(\overline G)}\le \|f\|_{L^\infty(G^c)}$ and  $\|{\mathcal
T}\psi\|_{C^1(\overline G)}\le C_J\|f\|_{L^\infty(G^c)}$.
Hence, by Schauder's fixed point theorem, ${\mathcal T}$ has a fixed
point $\hat\phi\in\mathcal K$. Then, the function
\[\phi(x)=\begin{cases}
\hat\phi(x),\quad&x\in\overline G,\\[6pt]
f(x),\quad&x\in \overline G^c.\end{cases}
\]
is the sought solution.

\noindent\textsf{Comparison. } Let
$\phi=\phi_1-\phi_2$. Then,
\[
\phi(x)\le \int_G J(x-y)\phi(y)\,dy.
\]
Assume for contradiction that $M=\max_{\overline G}\phi>0$. Let
$x_0\in\overline G$ be a point at which the maximum is attained
and let $G_i$ be the corresponding connected component. The set
$K=\{x\in\overline G_i:\phi(x)=M\}$ is nonempty and closed. On
the other hand,
\[
M=\phi(x_0)\le \int_G J(x_0-y)\phi(y)\,dy\le M\int_G J(x_0-y)\,dy\le M.
\]
This is only possible if $\phi(x)=M$ in
$\overline G \cap\mathop{\rm supp\,}J(x_0-\cdot)$. Since $J(0)>0$,
this implies that $\phi(x)=M$ in $\overline G\cap B(x_0;\delta)$
for some $\delta>0$; hence $K$ is an open subset of
$\overline G _i$. Since $\overline G_i$ is connected, we
deduce that $K=\overline G_i$, so that, $\phi\equiv M>0$ in
$\overline G_i$. But this implies that
$$
M=\phi(x_0)\le M\int_G J(x_0-y)\,dy<M
$$
for any $x_0\in\partial G_i$, a contradiction at any point where the density of $\overline G_i^c$ is positive.
\end{proof}

In order to prove that the function $\phi$ obtained in the limit procedure approaches 1 at infinity, we will compare the approximate solutions defined in bounded sets with a suitable barrier $z$. Hence we need to  estimate $Lz$.
\begin{lema}
\label{lemma:L.z}
Let $z(x)=(|x|^2+b)^{-\gamma}$,
 $\gamma>0$. There exists a number $b=b(J,\gamma,N)>0$ such that
\begin{equation}\label{potencia}
Lz(x)\le -2\alpha\,\gamma(|x|^2+b)^{-(\gamma+1)}(N-2- 2\gamma),
\end{equation}
where $\alpha$ is the constant defined in \eqref{eq:definition.alpha}.
\end{lema}
\begin{proof}
After some cumbersome, but easy,
computations, using Taylor's expansion and the radial symmetry of
$J$, we find that
\begin{equation*}
\begin{aligned}
Lz(x)=& -2\alpha\,\gamma(|x|^2+b)^{-(\gamma+1)}(N-2- 2\gamma)-4\alpha\,\gamma(\gamma+1)(|x|^2+b)^{-(\gamma+2)} b\\
&+ O\big((|x|^2+b)^{-(\gamma+2)}\big),
\end{aligned}
\end{equation*}
from where the result follows just choosing $b$ large enough.
\end{proof}

\begin{prop} There exists a unique solution to \eqref{stationary}. Moreover,
 $0\le \phi\le 1$ and there exists a
positive constant $K$ such that
\begin{equation}\label{cotas-fi}
1-\phi(x)\le \frac{K}{\big(|x|^2+1\big)^{(N-2)/2}},\qquad x\in\Omega.
\end{equation}
\end{prop}

\begin{proof}\textsf{Existence. }
Let $n_0$ in  $\mathbb{N}$ such that $\mathcal{H} \subset B(0;n_0)$. Consider
the sequence of functions $\{\phi_n(x)\}_{n\geq n_0}$, where
$\phi_n$ is the solution to
\begin{equation}\label{stationary-n}
\begin{cases}
L\phi_n(x)=0 \quad&\mbox{in }B(0;n) \setminus \mathcal{H},\\
\phi_n=0&\mbox{in }\mathcal{H},\\
\phi_n=1&\mbox{in }B(0;n)^c.\\
\end{cases}
\end{equation}
A simple comparison argument proves that $0\le \phi_{n}(x)\le 1$ for
all $n\ge n_0$, $x\in\mathbb{R}^N$. Hence,
$\phi_n(x)\ge\phi_{n+1}(x)$ in $(B(0;n))^c$, which yields, again by
comparison,   $0\leq \phi_{n+1}(x)\leq \phi_n (x)\leq 1$ for all
$n\ge n_0$, $x\in\mathbb{R}^N$. Therefore, the monotone limit
$$
\phi(x)= \lim_{n \rightarrow \infty} \phi_n(x)
$$
is well defined, and satisfies $0\le\phi\le1$. Letting $n\to\infty$
in \eqref{stationary-n}, we get
$$
L\phi(x)=0\mbox{ in}\quad\Omega,\qquad \phi=0\mbox{ in } \mathcal{H}.
$$

We still have to prove \eqref{cotas-fi},  which implies in
particular that $\phi(x)\to 1$ as $|x|\to\infty$. This is done
by comparison with the barrier $z$ described in Lemma~\ref{lemma:L.z} with $2\gamma=N-2$. Indeed, by taking
$K$ large enough we have,
$$
 Kz\ge1\ge 1-\phi_n \ \text{in }\mathcal{H},\quad Kz\ge
1-\phi_n =0 \ \text{in }(B(0;n))^c,\quad L (Kz)\le 0= L (1-\phi_n)\
\text{in }B(0;n) \setminus \mathcal{H}.
$$
Hence, $Kz\geq1-\phi_n$ in $\R^N$ and, passing to the limit, we
obtain \eqref{cotas-fi}.

\noindent\textsf{Uniqueness. } Let  $\phi$ and $\overline \phi$ be
two solutions to~\eqref{stationary}. We consider the family of
functions $\overline \phi_\lambda(x):=\overline \phi(x)+\lambda$.
Since both $\phi$ and $\overline\phi$ are bounded, there is a value
$\lambda_1$ such that $\overline \phi_{\lambda_1}\ge \phi$ in
$\R^N$. Let
$$
\lambda_0=\inf\{\lambda:\overline \phi_\lambda\ge \phi\mbox{ in }\R^N\}.
$$
We shall prove that $\lambda_0=0$, which means that $\overline
\phi\ge\phi$ in $\R^N$, from where uniqueness follows, since the
roles of $\overline \phi$ and $\phi$ can be interchanged.

Assume for contradiction that $\lambda_0>0$. Since  $\phi(x)$ and
$\overline \phi(x)$ tend to $1$ as $|x|\to\infty$, there exists
$R_1$ such that $\overline \phi_{\lambda_0}(x)>\phi(x)$ if $|x|\geq
R_1$. Moreover,  $\overline \phi_{\lambda_0}>\phi$ in $\mathcal{H}$. Hence, as
both $\overline \phi$ and $ \phi $ are continuous in $R^N\setminus
\mathcal{H}$, there exists $x_0$ in $\overline{B(0;R_1)}\setminus \mathcal{H}$ such that
$\overline \phi_{\lambda_0}(x_0)=\phi(x_0)$. The set
$$
S=\{ x \in \overline{B(0;R_1)} \setminus \mathcal{H}: \overline
\phi_{\lambda_0}(x)=\phi(x)\}
$$
is closed and non-empty. If $x_1$ is a point in  $\partial S$ where the density of $S^c$ is
positive, then
 $\overline \phi_{\lambda_0}(x_1)=\phi(x_1)$ but $J*\overline
\phi_{\lambda_0}(x_1)>J*\phi(x_1)$. Since this is a contradiction we
get that $\lambda_0=0$.
\end{proof}

\subsection{Asymptotic properties}
We now obtain  the decay at infinity of the
derivatives  of $\phi$. In the sequel we let $\psi=1-\phi$.

\begin{prop} Given $s\in\N_0$, there exists a constant $C_s$ such that, for every multi-index $\beta\in\N_0^N$ of order $|\beta|=s$,
\begin{equation}\label{derivative-psi}
|D^\beta\psi(x)|\le \frac {C_s}{|x|^{N-2+s}},\qquad x\in\Omega.
\end{equation}
\end{prop}
\begin{proof}
The function $\phi$ is the unique bounded solution to the evolution problem~\eqref{problem} with initial data $\phi$. Hence, using the representation formula \eqref{eq:representation.formula} with $t_0=0$ and  $u(x,t_0)=\phi(x)$, we get,
\[
\begin{aligned}
\phi(x)&=\frac1{1-e^{-t}}\int_{\mathbb{R}^N} \omega(x-y,t)\phi(y)\,dy-\X_\mathcal{H}(x)(J*\phi)(x)\\
&-\frac1{1-e^{-t}}\int_0^t\int_\mathcal{H} \omega(x-y,t-s)(J*\phi)(y)\,dy\,ds.
\end{aligned}
\]
Therefore,   since $\int_{\mathbb{R}^N}\omega(x,t)\,dx=1-e^{-t}$, we have
\begin{equation}
\label{eq:phi}
\begin{aligned}
\psi(x)&=\frac1{1-e^{-t}}\int_{\mathbb{R}^N}\omega(x-y,t)\psi(y)\,dy+\X_\mathcal{H}(x)(J*\phi)(x)\\\
&+\frac1{1-e^{-t}}\int_0^t\int_{\mathcal{H}} \omega(x-y,t-s)(J*\phi)(y)\,dy\,ds.
\end{aligned}
\end{equation}
Given $k>0$, we denote
$$
\phi^k(x)=k^{N-2}\phi(kx),\quad \psi^k(x)=k^{N-2}\psi(kx),\quad
J_k(x)=k^NJ(kx),\quad \omega_k(x,t)=k^N(kx,k^2t),
$$
and
$\mathcal{H}_k=\{x\in\R^N:kx\in \mathcal{H}\}$. If we scale
equation~\eqref{eq:phi} according to these recipes, we get
\[
\begin{aligned}
\psi^k(x)&= \frac1{1-e^{-t}}\int_{\mathbb{R}^N} \omega_k(x-y,k^{-2}t)\psi^k(y)\,dy+\X_{\mathcal{H}_k}(x)(J_k*\phi^k)(x)\\
&+\frac1{1-e^{-t}}\int_0^t\int_{\mathcal{H}_k}\omega_k(x-y,k^{-2}(t-s))(J_k*\phi^k)(y)\,dy\,ds.
\end{aligned}
\]
Let $\delta>0$ be given. There exists a value $k_\delta$ such that
$\X_{\mathcal{H}_k}(x)=0$ if $|x|\ge\delta$ and $k\ge k_\delta$. Therefore,
\[
\begin{aligned}
D^\beta\psi^k(x)&=\underbrace{\frac1{1-e^{-t}}\int_{\mathbb{R}^N}D^\beta \omega_k(x-y,k^{-2}t)\psi^k(y)\,dy}_{\mathcal{A}}\\
&+\underbrace{\frac1{1-e^{-t}}\int_0^t\int_{\mathcal{H}_k}D^\beta \omega_k(x-y,k^{-2}(t-s))(J_k*\phi^k)(y)\,dy\,ds}_{\mathcal{B}},
\end{aligned}
\]
if $|x|\ge\delta$ and $k\ge k_\delta$.

We start by estimating $\mathcal{B}$. If $y\in \mathcal{H}_k$ and $|x|\ge\delta$,
then $|x-y|\ge\delta/2$ if $k\ge k_\delta$. On the other
hand, since $0\le\phi^k(x)\le k^{N-2}$, we also have $0\le
J_k*\phi^k\le k^{N-2}$. So, recalling that $|\mathcal{H}_k|=Ck^{-N}$ and the
bounds in \eqref{barriers} -- that are scaling invariant --  we get
\[
\begin{aligned}
|\mathcal{B}|&\le  \frac{k^{N-2}}{1-e^{-t}}\int_0^t\int_{\mathcal{H}_k}\frac {Ck^{-2}(t-s)}{|x-y|^{N+s+2}}\,dy\,ds\\
&\le C_{\delta,N}\frac{k^{-4}}{1-e^{-t}}\int_0^t (t-s)\,ds=C_{\delta,N}\frac{k^{-4}}{1-e^{-t}}t^2\le C_{\delta,N}\quad\mbox{if}\quad t=k^2.
\end{aligned}
\]

Now we estimate $\mathcal{A}$ with $t=k^2$. Recall that $0\le\psi(x)\le
C/|x|^{N-2}$. Therefore we have,  using again
\eqref{barriers} and \eqref{integral-estimates},
\[
\begin{aligned}
|\mathcal{A}|&\le \frac1{1-e^{-k^2}}\int_{|y|<\delta/2}|D^\beta \omega_k(x-y,1)|\frac C{|y|^{N-2}}\,dy+\frac{C_{\delta,N}}{1-e^{-k^2}}\int_{|y|\ge\delta/2} |D^\beta \omega_k(x-y,1)|\,dy\\
&\le \frac{C_{\delta,N}}{1-e^{-k^2}}\int_{|y|<\delta/2}\frac{dy}{|y|^{N-2}}+
\frac{C_{\delta,N}}{1-e^{-k^2}}\le C_{\delta,N}.
\end{aligned}
\]

In conclusion, given $\delta>0$ there exists $k_\delta$ such that,
if $k\ge k_\delta$ and $|x|\ge\delta$, then
$|D^\beta\psi^k(x)|\le C_{\delta,N}$.

Let us take $|x|=1$, $C_1=C_{1,N}$, $y=kx$ so that $k=|y|$. Then, we
get
\[
|y|^{N-2+s}|D^\beta\psi(y)|=k^{N-2+s}|D^\beta\psi(kx)|=|D^\beta\psi^k(x)|\le
C_1\quad\mbox{if}\quad |y|=k\ge k_1.
\]
This proves estimate \eqref{derivative-psi} except for a bounded set. However, $\phi$ is smooth in $\Omega$, and hence the estimate is obviously true for bounded sets.
\end{proof}

\begin{rem}
The proof of the bound~\eqref{derivative-psi} only requires $J\in C^{s+1}$.
\end{rem}

Though the second derivatives of $\psi$ decay like $O(|x|^{-N})$ as $|x|\to\infty$, a special combination, the Laplacian, decays faster. In particular, $\Delta\psi$ is integrable at infinity.

\begin{coro}\label{Lpsi} There exists a constant $C$ such that $|\Delta \psi(x)|\le \frac C{|x|^{N+2}}$ in $\Omega$.
\end{coro}
\begin{proof}
 Using Taylor's expansion we get $L\psi(x)=\alpha\Delta\psi(x)+ O\big(\max_{|\beta|=4}\|D^\beta\psi\|_{L^\infty(B_1(x))}\big)$. Since $\psi$ is $L$-harmonic, the result follows from the estimate~\eqref{derivative-psi} for the fourth order derivatives.
\end{proof}

\begin{rem}
The proof of Corollary~\ref{Lpsi} only requires $J\in C^4$. If we weaken this assumption and only require $J\in C^3$, we get $|\Delta\psi(x)|\le C/|x|^{N+1}$, and the Laplacian of $\psi$ is still integrable at infinity.
\end{rem}

We can take profit of the integrability of $\Delta\psi$ at infinity to obtain the precise asymptotic behavior of $\psi$.

\begin{prop} There exists a constant $C^*$ such that for every multi-index $\beta\in\mathbb{N}_0^N$ of order $|\beta|=s$,
\begin{equation}\label{asym-psi}
|x|^{N-2+s}\Big|D^\beta\psi(x)-D^\beta \psi^*(x)\Big|\to0\quad\mbox{as}\quad|x|\to\infty,\qquad \psi^*(x)=C^*|x|^{2-N}.
\end{equation}
\end{prop}
\begin{proof}\textsf{Scaling and compactness. }
From~\eqref{derivative-psi} we get that the scaled functions
$\psi^k$ satisfy $|D^\beta\psi^k(x)|\le C/|x|^{N-2+s}$ if
$|\beta|=s$. Hence, thanks to Arzel\`{a}-Ascoli's theorem, for any
sequence $\{k_n\}_{n\in\mathbb{N}}$ there exist a subsequence, that
we name again $\{k_n\}_{n\in\mathbb{N}}$, and a function $\psi^*$
such that, for any multi-index $\beta$,
\begin{equation}
\label{eq:convergence.derivatives}
D^\beta \psi^{k_{n}}\to
D^\beta\psi^*\quad\text{ uniformly on compact sets of }\R^N\setminus\{0\}.
\end{equation}
Moreover, $\psi^*$ inherits from the functions $\psi^k$ the estimate
\begin{equation}
\label{eq:estimate.bar.psi}
0\le \psi^*(x)\le C/|x|^{N-2}.
\end{equation}

\noindent\textsf{Identification of the limit along subsequences. }
We consider the scaled operator $L_k$ defined by
$L_kv=k^2\big(J_k*v-v\big)$. Performing a Taylor expansion, it is
easy to check that $L_k\varphi\to\alpha\Delta \varphi$ uniformly for
any $\varphi\in C^\infty_{\text{c}}(\mathbb{R}^N)$.

Let now $\varphi\in C^\infty_{\text{c}}(\mathbb{R}^N\setminus\{0\})$.
We write
\begin{equation}
\label{eq:convergence.to.harmonic}
\int_{\mathbb{R}^N}\psi^*\alpha\Delta\varphi=\int_{\mathbb{R}^N}\psi^kL_k\varphi+\int_{\mathbb{R}^N}(\psi^*-\psi^k)\alpha\Delta\varphi
+\int_{\mathbb{R}^N}\psi^k(\alpha\Delta\varphi-L_k\varphi).
\end{equation}
Since $L_k\psi^k=0$, we have
$\int_{\mathbb{R}^N}\psi^kL_k\varphi=\int_{\mathbb{R}^N}L_k\psi^k\varphi=0$.
Besides, using~\eqref{derivative-psi}, for $k$ large enough we have
$$
\left|\int_{\mathbb{R}^N}\psi^k(\alpha\Delta\varphi-L_k\varphi)\right|\le
C\int_{\mathbb{R}^N\cap\big(\mathop{\rm supp\,}
\varphi+B_{1/k}(0)\big)}|\alpha\Delta\varphi-L_k\varphi|\to0\quad\text{as
}k\to\infty.
$$
Therefore, taking $k=k_n$ in~\eqref{eq:convergence.to.harmonic} and
letting $n\to\infty$, we get
$\int_{\mathbb{R}^N}\psi^*\alpha\Delta\varphi=0$, which implies
that $\Delta\psi^*=0$ in $\R^N\setminus \{0\}$. Since $
|x|^{N-2}\psi^*$ is bounded, see estimate~\eqref{eq:estimate.bar.psi}, we conclude that
there exists a constant $C^*$ such
that $\psi^*(x)= C^*/|x|^{N-2}$.

\noindent\textsf{Uniqueness of the limit. }
We next prove that the constant $C^*$ is independent of the sequence $\{k_{n}\}_{n\in\mathbb{N}}$.

A direct computation shows that
\[
\int_{\partial B(0;R)}\frac{\partial\psi^*}{\partial r}\,dS=C^* (2-N)\, {\mathcal H}^{N-1}\big(\partial B(0;1)\big)\quad \text{for all }R>0.
\]
On the other hand, if $\mathcal{H}\subset\subset B_{R_0}$,
\[
\begin{aligned}
\int_{\partial B(0;R)}\frac{\partial\psi}{\partial r}\,dS=
\int_{\partial B(0;R_0)}\frac{\partial\psi}{\partial r}\,dS
+ \int_{B_R\setminus B_{R_0}}\Delta\psi(x)\,dx \quad\text{for all }R>R_0.
\end{aligned}
\]
Since $\int_{B_R\setminus B_{R_0}}\Delta\psi(x)\,dx$ has a finite limit as $R\to\infty$, see Corollary \ref{Lpsi},
$\int_{\partial B(0;R)}\frac{\partial\psi}{\partial r}\,dS$ also has a limit, that we denote by $\mu$.

Let $y=kx$. On one hand, we have  $|y|^{N-1}\nabla\psi(y)=|x|^{N-1}\nabla\psi^k(x)$ and on the other, $|y|^{N-1}\nabla \psi^*(y)=|x|^{N-1}\nabla \psi^*(x)$. Therefore, taking $k=k_n$ and $|x|=1$, which implies $|y|=k_{n}$, and using~\eqref{eq:convergence.derivatives}, we get
\[
|y|^{N-1}\big|\nabla\psi(y)-\nabla\psi^*(y)\big|=\big|
\nabla\psi^{k_{n}}(x)-\nabla\psi^*(x)\big|
<\ep\quad\text{if }|y|=k_{n}\text{ and }n\text{ is large}.
\]
Hence,
\[
\Big|\int_{\partial B(0;k_{n})}\frac{\partial\psi}{\partial r}\,dS-C^* (2-N)\,\H^{N-1}\big(\partial B(0;1)\big)\Big|=
\Big|\int_{\partial B(0;k_{n})}\left(\frac{\partial\psi}{\partial r}-\frac{\partial\psi^*}{\partial r}\right)\,dS\Big|
<C\ep
\]
if $n$ is large. We conclude that, $C^*=\mu\,\big((2-N)\,\H^{N-1}\big(\partial B(0;1)\big)\big)^{-1}$.

\noindent\textsf{Rephrasing the limit. } Take  $y=kx$ and $|x|=1$. In the above steps we have proved that
\[
|y|^{N-2+s}D^\beta\psi(y)=k^{N-2+s}D^\beta\psi(kx)=D^\beta\psi^k(x)\to D^\beta \psi^*(x)\quad\mbox{as}\quad |y|=k\to\infty.
\]
Since $D^\beta \psi^*(x)=|y|^{N-2+|\beta|}D^\beta\psi^*(y)$ (remember that $|x|=1$), this is just \eqref{asym-psi}.
\end{proof}

\begin{rem}
To get~\eqref{asym-psi} it is enough to have $J\in C^{\max(3,s)}$. We need at least $J\in C^3$, no matter the order of the derivative we are considering, in order for $\Delta\psi$ to be integrable at infinity, which is one of the  main ingredients of the proof.
\end{rem}

\section{Conservation law and asymptotic mass}
\setcounter{equation}{0}

We start our study of the evolutionary problem \eqref{problem} with a conservation law that will be of the most importance in
order to establish the behavior of the solution
in the far field.

\begin{prop} Let $u$ be the solution to \eqref{problem} and $\phi$ the solution to \eqref{stationary}. Then,
\begin{equation}\label{ley-conservacion}
\int_{\mathbb{R}^N} u(x,t)\phi(x)\,dx=\int_{\mathbb{R}^N}
u_0(x)\phi(x)\,dx:=M^*\quad\mbox{for every }t>0.
\end{equation}
\end{prop}
\begin{proof} Since $u(x,t)=\phi(x)=0$ for $x\in \mathcal{H}$, we have
$$
\begin{aligned}
\frac{d}{dt}\int_{\mathbb{R}^N}
u(x,t)\phi(x)\,dx&=\int_{\mathbb{R}^N} u_t(x,t)\phi(x)\,dx\\&=
\int_{\mathbb{R}^N}\int_{\mathbb{R}^N}
J(x-y)\big(u(y,t)-u(x,t)\big)\phi(x)\,dy\,dx\\&=
\int_{\mathbb{R}^N}\int_{\mathbb{R}^N}
J(x-y)\big(\phi(y)-\phi(x)\big)u(x,t)\,dy\,dx\\&=0.
\end{aligned}
$$
\end{proof}

Though mass is not conserved, there is a non-trivial asymptotic
mass, which coincides with the conserved quantity $M^*$.
\begin{prop}\label{mass} Let $u$ be the solution to \eqref{problem} and $M^*=\int_{\mathbb{R}^N} u_0(x)\phi(x)\,dx$.
The mass of the solution at time $t$,
$M(t)=\displaystyle\int_{\mathbb{R}^N}u(x,t)\, dx$, satisfies $
M(t)\to M^*$ as $t\to\infty$.
\end{prop}
\begin{proof}
In fact,
\[\begin{aligned}
\left|\int_{\mathbb{R}^N} u(x,t)\,dx-M^*\right|&=\int_{\mathbb{R}^N} u(x,t)\big(1-\phi(x)\big)\,dx\\
&\le \int_{|x|<\eta\sqrt t}u(x,t)\,dx+\int_{|x|\ge\eta\sqrt t}u(x,t)\big(1-\phi(x)\big)\,dx\\
&\le C t^{-N/2}(\eta\sqrt t)^N+C\frac1{(\eta\sqrt t)^{N-2}}\int_{\mathbb{R}^N} u(x,t)\,dx\\
&\le C\eta^N+C\|u_0\|_{L^1(\mathbb{R}^N)}\frac1{(\eta\sqrt t)^{N-2}}.
\end{aligned}
\]
Hence, choosing $\eta$ so that $C\eta^N<\ep$ and taking
$\limsup_{t\to\infty}$ we get,
\[
\limsup_{t\to\infty}\left|\int_{\mathbb{R}^N}
u(x,t)\,dx-M^*\right|\le \ep.
\]
We conclude by letting $\ep\to0$.
\end{proof}

\section{Far-field limit}
\setcounter{equation}{0}

In this section we study the large time behavior of the solution
to~\eqref{problem} in sets of the form $\{|x|^2\ge \delta {t}\}$,
with $\delta>0$.
\begin{teo}\label{teo1} Let $u$ be the solution of \eqref{problem} and $M^*$ the asymptotic mass given in
\eqref{ley-conservacion}.
Then, for every $\delta>0$,
\begin{equation}
\label{estimate-1}
t^{N/2}\|u(x,t)-M^*\Gamma_\alpha(x,t)\|_{L^\infty(\{|x|^2\ge\delta
t\})}\to0\quad\mbox{as}\quad t\to\infty.
\end{equation}
\end{teo}
\begin{proof} Let $t_0>0$ and $t\ge 2t_0$. From the representation formula \eqref{eq:representation.formula} we have
\[
\begin{aligned}
t^{N/2}|u(x,t)- M^*\Gamma_\alpha(x,t)|&\le \underbrace{t^{N/2}e^{-(t-t_0)}u(x,t_0)}_{I_1}\\
&\ +\underbrace{t^{N/2}\left|\int_{\mathbb{R}^N} \omega(x-y,t-t_0)u(y,t_0)\,dy-M^*\Gamma_\alpha(x,t)\right|}_{I_2}\\
&\ +\underbrace{t^{N/2}\int_{t_0}^te^{-(t-s)}\X_\mathcal{H}(x)(J*u(\cdot,s))(x)\,ds}_{I_3}\\
&\ +\underbrace{t^{N/2}\int_{t_0}^t\int_{\mathbb{R}^N} \omega(x-y,t-s)\X_\mathcal{H}(y)(J*u(\cdot,s))(y)\,dy\,ds}_{I_4}.
\end{aligned}
\]
Let $\delta>0$ and $|x|^2\ge \delta t$. We have the following estimates for the quantities $I_i$:

\noindent $\blacktriangleright$ $
I_1\le C t^{N/2}e^{-t/2}\to 0$ as $t\to\infty$.

\noindent $\blacktriangleright$
The estimate for $I_2$ is the most involved, and is split in three parts,
 \[
\begin{aligned}
I_2&\le \underbrace{t^{N/2}\int_{\mathbb{R}^N}\left| \omega(x-y,t-t_0)-\Gamma_\alpha(x-y,t-t_0)\right|u(y,t_0)\,dy}_{I_{2,1}}\\
&\ +\underbrace{t^{N/2}\left|\int_{\mathbb{R}^N}\Gamma_\alpha(x-y,t-t_0)u(y,t_0)\,dy-M(t_0)\Gamma_\alpha(x,t)\right|}_{I_{2,2}}\\
&\ +\underbrace{t^{N/2}\Gamma_\alpha(x,t)|M(t_0)-M^*|}_{I_{2,3}}.
\end{aligned}
\]
\begin{itemize}
\item Using \eqref{estima-W},
$$
\begin{aligned}
I_{2,1}&\le t^{N/2}\|u(\cdot,t_0)\|_{L^1(\mathbb{R}^N)}\|\omega(\cdot,t-t_0)-\Gamma_\alpha(\cdot,t-t_0)\|_{L^\infty(\mathbb{R}^N)}\\
&= t^{N/2}\|u_0\|_{L^1(\mathbb{R}^N)}o\big((t-t_0)^{-\frac{N}2}\big).
\end{aligned}
$$
Thus, $\lim_{t\to\infty}I_{2,1}=0$.

\item  The well-known asymptotics for (local) caloric functions imply $\lim_{t\to\infty} I_{2,2}=0$.

\item By Proposition \ref{mass},  if $t_0$ is large enough,
$I_{2,3}\le  C_\alpha|M(t_0)-M^*|<\ep$.

\end{itemize}

\noindent $\blacktriangleright$ For $t_0$ large enough,  $x\notin \mathcal H$. Hence, $I_3=0$.

\noindent $\blacktriangleright$ If $t_0$ is large enough and $y\in \mathcal{H}$, $|x-y|\ge|x|/2$. On the other hand, by \eqref{decaimiento}
$0\le u(x,s)\le C s^{-N/2}$ and therefore, $0\le (J*u(\cdot,s))(x)\le C s^{-N/2}$.
Thus, using the space-time bound \eqref{barriers} for $\omega$,
$$
I_4\le C t^{N/2}\int_{t_0}^t\int_\mathcal{H}\frac{t-s}{|x-y|^{N+2}}\,s^{-N/2}\,dy\,ds
\le C\Big(\frac{t}{|x|^2}\Big)^{\frac N2+1} t_0^{-\frac N2+1}\le C\delta^{-\frac N2-1}\,t_0^{-\frac N2+1}<\ep
$$
if $t_0$ is large enough.

Gathering all the above estimates, we get
\[
\limsup_{t\to\infty}\big\|t^{N/2}u(x,t)-t^{N/2}M^*\Gamma_\alpha(x,t)
\big\|_{L^\infty(|x|^2\ge\delta t)}\le2\ep.
\]
We finally let $\ep\to0$.
\end{proof}

As a corollary of the outer behavior, we  can do better in the
proof of Proposition~\ref{mass} to obtain the rate of decay of the
mass to its asymptotic limit.
\begin{coro}\label{cor.masa}
Under the assumptions of Proposition~\ref{mass}, we have
\begin{equation}
M(t)=M^* + K\,t^{-\frac{N-2}2} + o(t^{-\frac{N-2}2}),
\end{equation}
where $K=C^*M^*\int_{\mathbb{R}^N}U_{\alpha}(\xi)|\xi|^{2-N}\,d\xi,$
$\displaystyle C^*=\lim_{|x|\to\infty}(1-\phi(x))|x|^{N-2}$.
\end{coro}

\begin{proof}
We have from the conservation law~\eqref{ley-conservacion} that
$$
t^{\frac{N-2}2}(M(t)-M^*)=t^{\frac{N-2}2}\int_{\mathbb{R}^N}u(x,t)(1-\phi(x))\,dx.
$$
In order to estimate this last integral we split it into two parts,
$$
\begin{aligned}
\displaystyle t^{\frac{N-2}2}&\int_{\mathbb{R}^N}
u(x,t)(1-\phi(x))\, dx=
\\ &\displaystyle \underbrace{ t^{\frac{N-2}2}\int_{\{|x|\le
\delta t^{1/2}\}} u(x,t)(1-\phi(x))\, dx}_{ I_1}
+\underbrace{t^{\frac{N-2}2}\int_{\{|x|\ge \delta t^{1/2}\}}
u(x,t)(1-\phi(x))\, dx}_{I_2}
\end{aligned}
$$
for $\delta>0$ and $t$ big enough. Using the estimate
$0\le1-\phi(x)\le c_1|x|^{2-N}$ and the bound $0\le u(x,t)\le C t^{-N/2}$ we get $0\le I_1\le C\delta^2$. For $I_2$, we perform the
change of variables $x=\xi t^{1/2}$, and get
$$\begin{aligned}
\lim_{t\to\infty} I_2&=\int_{\{|\xi|\ge \delta \}} \lim_{t\to\infty}
\frac{|\xi|^{N-2}\left(1-\phi(\xi
t^{1/2})\right)t^{\frac{N-2}2}}{|\xi|^{N-2}}u(\xi
t^{1/2},t)t^{N/2}\, d\xi\\&=C^*M^*\int_{\{|\xi|\ge \delta \}}
U_{\alpha}(\xi)|\xi|^{2-N}\,d\xi.
\end{aligned}
$$
Hence
$$
\begin{aligned}
C^*M^*&\int_{\{|\xi|\ge \delta \}}
U_{\alpha}(\xi)|\xi|^{2-N}\,d\xi\le\liminf_{t\to\infty}t^{\frac{N-2}2}(M(t)-M^*)
\\
&\le \limsup_{t\to\infty} t^{\frac{N-2}2}(M(t)-M^*)\le
C\delta^2+C_*M^*\int_{\{|\xi|\ge \delta \}}
U_{\alpha}(\xi)|\xi|^{2-N}\,d\xi,
\end{aligned}
$$
from where the result follows just by letting $\delta\to 0$.
\end{proof}

\begin{rem} The amount of mass, $M_L(u)$, lost in the
evolution is given by
$$
M_L(u):= \int_{\mathbb{R}^N} u_0(x)\, dx - \lim_{t\to\infty}
\int_{\mathbb{R}^N} u(x,t)\,dx=  \int_{\mathbb{R}^N}
(1-\phi(x))u_0(x)\, dx>0.
$$
Therefore, the influence of the hole structure is felt at the
asymptotic level through the projection of the initial data onto
$\psi:=1-\phi$, which represents in this way the dissipation
capacity of $\mathcal{H}$.

Notice that $\psi$ is the $L$-harmonic function defined in $\Omega$
that takes value 1 on $\mathcal{H}$ and 0 at infinity. Hence, by analogy with the standard (local) potential theory, we may say that $\psi$ is the function that  measures the $L$-capacity of $\mathcal{H}$, by means of the
formula
\begin{equation*}
\label{capacidad}
 \mathop{\rm cap }\nolimits(\mathcal{H})=\inf_{\{u\ge1\text{ on }\H\}}\int_{\Omega}\int_\Omega J(x-y)(u(x)-u(y))^2\,dx\,dy.
\end{equation*}
\end{rem}

\section{Near-field limit}
\setcounter{equation}{0}

In view of Theorem~\ref{teo1}, what is left to complete the proof of
Theorem~\ref{thm:main} is to show that the limit~\eqref{result} is
valid uniformly in sets of the form $\{|x|^2<\delta t\}$ for some
$\delta>0$. This is the goal of this section.
\begin{teo}\label{thm:inner.behavior}
Let $u$ be the solution to \eqref{problem}. There exists a value
$\delta>0$ such that
\begin{equation}
\label{eq:thm.inner.behavior}
\lim_{t\to\infty}t^{N/2}\|u(x,t)-M^*
\phi(x)\Gamma_\alpha(x,t)\|_{L^\infty(\{|x|^2<\delta t\})}=0.
\end{equation}
\end{teo}

Instead of~\eqref{eq:thm.inner.behavior}, we will prove
\begin{equation}\label{result2}
\lim_{t\to\infty}t^{N/2}\|u(x,t)-M^*\phi(x)\omega(x,t)\|_{L^\infty(\{|x|^2<\delta
t\})}=0,
\end{equation}
which turns out to be equivalent, thanks to~\eqref{estima-W} with $|\beta|=0$.
Notice also that, using again~\eqref{estima-W}, together with the asymptotic behavior of $\phi$, we may rewrite~\eqref{estimate-1} in a similar way,
\begin{equation}\label{eq:outer.with.omega}
\lim_{t\to\infty}t^{N/2}\|u(x,t)-M^*\phi(x)\omega(x,t)\|_{L^\infty(\{|x|^2\ge\delta
t\})}=0.
\end{equation}

In order to prove~\eqref{result2} we will construct suitable barriers approaching the asymptotic limit as $t$ goes to infinity.
We choose
$\kappa\in(0,\min\{1,N-2\})$, $\gamma\in(0,(N-2-\kappa)/2)$, and
then define, for any $K_\pm>0$,
\begin{equation}
\label{eq:def.parabolic.barriers} v_\pm(x,t)=\phi(x)\omega(x,t)\pm
K_\pm t^{-\frac {N+\kappa}{2}}z(x),
\end{equation}
with $z$  as in Lemma~\ref{lemma:L.z}, $z(x)=(|x|^2+b)^{-\gamma}$.
We take $b$ large enough, so that \eqref{potencia} is satisfied.
\begin{lema}
\label{lem:v.plus.minus}
Let $R>0$ and $v_\pm$ as above. There exists a value $\delta_*>0$, independent of $K_\pm$, such that for all $\delta\in(0,\delta^*)$, $K_\pm\ge1$, we have
$$
\partial_tv_+-Lv_+\ge 0,\quad \partial_t v_{-}-Lv_-\le0, \qquad R^2\le |x|^2\le \delta t.
$$
\end{lema}
\begin{proof}
On one hand,
$$
\partial_t v_+(x,t)=\phi(x)\partial_t \omega(x,t)-\frac {N+\kappa}{2t}K_+t^{-\frac {N+\kappa}{2}}z(x).
$$
On the other hand, using that $\phi$ is $L$-harmonic and \eqref{potencia}, we get
$$
\begin{aligned}
Lv_+(x,t)\leq &\phi(x)L\omega(x,t)+\int_{\mathbb{R}^N} J(x-y)\big(\phi(y)-\phi(x)\big)\big(\omega(y,t)-\omega(x,t)\big)\,dy\\
&\qquad-\frac{2\alpha\gamma
(N-2-2\gamma)}{|x|^2+b}K_+t^{-\frac{N+\kappa}{2}}z(x).
\end{aligned}
$$
Hence, since $\omega$ solves~\eqref{eq-W} and $\phi\ge0$,
$$
\begin{aligned}
\partial_t v_+-Lv_+&\geq-\underbrace{\int_{\mathbb{R}^N} J(x-y)\big|\phi(y)-\phi(x)\big|\big|\omega(y,t)-\omega(x,t)\big|\,dy}_{\mathcal{A}}\\
&+K_+t^{-\frac{N+\kappa}{2}}z(x)\underbrace{\left(\frac{2\alpha\gamma
(N-2-2\gamma)}{|x|^2+b}-\frac {N+\kappa}{2t}\right)}_{\mathcal{B}}.
\end{aligned}
$$

\noindent\emph{Estimate for $\mathcal{A}$. } We first notice that
there is a constant $C>0$ such that
$$
|\phi(x)-\phi(y)|\le \frac{C}{|x|^{N-1}} \qquad\text{if }|x-y|\le1.
$$
Indeed, since $|\phi(x)-\phi(y)|\le 1$, this is obviously true for
$|x|\le2$, with $C=2^{N-1}$. On the other hand, using Taylor's
expansion,
\begin{equation}
\label{eq:difference.phi}
|\phi(x)-\phi(y)|\le \max_{|\beta|=1}\max_{\xi\in
B(x;1)}|D^\beta\phi(\xi)|\qquad\text{if }|x-y|\le1.
\end{equation}
The result for $|x|\ge2$ now follows from the
estimate~\eqref{derivative-psi} with $|\beta|=1$, since $|\xi|\ge
|x|/2$ for all $\xi\in B(x;1)$ whenever $|x|\ge2$.

Combining~\eqref{eq:difference.phi} with~\eqref{decaimiento-W}, we
conclude that there is a constant $C>0$ such that
$$
\mathcal{A}\le
\frac{Ct^{-\frac{N+1}2}}{|x|^{N-1}}=\frac{Ct^{-\frac{N+\kappa}2}(|x|^2/t)^{\frac{1-\kappa}2}}{|x|^{N-\kappa}}
\le\frac{Ct^{-\frac{N+\kappa}2}\delta^{\frac{1-\kappa}2}}{|x|^{N-\kappa}}\qquad\text{if
}|x|^2\le\delta t.
$$

\noindent\emph{Estimate for $\mathcal{B}$. } If
$R^2\le|x|^2\le\delta t$ and $
\delta\le\frac{\alpha\gamma(N-2-2\gamma)R^2}{(N+\kappa)(R^2+b)}$, we
have
$$
\frac{N+\kappa}{2t}\le\frac{\delta(N+\kappa)}{|x|^2}=\frac{\delta(N+\kappa)(R^2+b)}{R^2|x|^2+b|x|^2}\le
\frac{\delta(N+\kappa)(R^2+b)}{R^2(|x|^2+b)}\le
\frac{\alpha\gamma(N-2-2\gamma)}{|x|^2+b}.
$$
Thus, with this choice of $\delta$,
$$
\mathcal{B}\ge\frac{\alpha\gamma(N-2-2\gamma)}{|x|^2+b}.
$$

The above estimates for $\mathcal{A}$ and $\mathcal{B}$ yield, if $K_+\ge1$,
$$
\partial_t v_+-Lv_+\ge \frac{t^{-\frac{N+\kappa}2}}{(|x|^2+b)^{\gamma+1}}\left(\alpha
\gamma(N-2-2\gamma)-\frac{C\delta^{\frac{1-\kappa}2}(|x|^2+b)^{\gamma+1})}{|x|^{N-\kappa}}\right).
$$
On the other hand, since $2(\gamma+1)\le N-\kappa$,
$$
\frac{(|x|^2+b)^{\gamma+1}}{|x|^{N-\kappa}}\le \frac{(R^2+b)^{\gamma+1}}{R^{N-\kappa}}.
$$
Therefore, taking $\delta$ small enough, we get
$$
\partial_tv_+-Lv_+\ge 0, \qquad R^2\le |x|^2\le \delta t
$$
for all $K_+\ge1$.
Notice that the threshold value of $\delta$ does not depend on $K_+$.

An analogous argument leads to the statement concerning $v_-$ since $e^{-t}J(x)\le C t^{\frac{N+1}2}|x|^{1-N}$.
\end{proof}

\begin{proof}[Proof of Theorem~\ref{thm:inner.behavior}. ]
Let $R>0$ such that $B(0;R)\subset \mathcal{H}$. As a consequence of Lemma~\ref{lem:v.plus.minus}, we know that there is a value $\delta>0$ such that
$$
\partial_tv_+-Lv_+\ge 0, \qquad x\in\Omega\cap\{ |x|^2\le \delta t\},\ t>0.
$$
Let now $\ep>0$. By \eqref{eq:outer.with.omega},~\eqref{estima-W} with $|\beta|=0$, and the asymptotic behavior of $\phi$, there exists $t_0>0$ such that
$
u(x,t)\le (1+\ep)M^*\phi(x)\omega(x,t)$ if $(\delta t)^{1/2}\le |x|\le(\delta t)^{1/2}+1$, $t\ge t_0$.
Hence,
for any $K_+\ge0$,
$$
u(x,t)\le (1+\ep)M^* v_+(x,t)\quad\mbox{if }(\delta t)^{1/2}\le |x|\le(\delta t)^{1/2}+1,\ t\ge t_0.
$$
On the other hand, it is obvious that there exists $K_+\ge1$ such that $u(x,t_0)\le (1+\ep)M^* v_+(x,t_0)$ if $|x|\le (\delta t_0)^{1/2}+1$. Finally,
$
u(x,t)=0\le (1+\ep)M^*v_+(x,t)$, $x\in\mathcal{H}$, $t\ge t_0$.
Therefore, by comparison, see Theorem~\ref{thm:contraction.general} and Remark~\ref{rem:comparison.general},
$$
u(x,t)\le (1+\ep)M^* v_+(x,t),\quad|x|\le  (\delta t)^{1/2},\ t\ge t_0.
$$
Hence, using the decay estimate~\eqref{decaimiento-W}, if $|x|\le  (\delta t)^{1/2}$, $t\ge t_0$, we have
$$
\begin{array}{rcl}
t^{N/2}\left(u(x,t)-M^*\phi(x)\omega(x,t)\right)&\le& \ep M^*\phi(x)t^{N/2}\omega(x,t)+(1+\ep)M^*K_+t^{-\frac\kappa2}z(x)\\
&\le&\ep M^*C+(1+\ep)M^*K_+t^{-\frac\kappa2}b^{-\gamma}.
\end{array}
$$
Letting $t\to\infty$ and then $\ep\to0$, we conclude that
$\limsup\limits_{t\to\infty}t^{N/2}\left(u(x,t)-M^*\phi(x)\omega(x,t)\right)\le 0$.

An analogous argument shows that
$$
 (1-\ep)M^* v_-(x,t)\le u(x,t),
 \quad|x|\le  (\delta t)^{1/2},\ t\ge t_0.
$$
Hence, $\liminf\limits_{t\to\infty}t^{N/2}\left(u(x,t)-M^*\phi(x)\omega(x,t)\right)\ge0$.
\end{proof}

As a corollary, we obtain the  behavior on compact sets.
\begin{coro} Let  $u$ be the solution to \eqref{problem}. Then,
\begin{equation}
\label{eq:inner.limit}
t^{N/2}u(x,t)\to \frac{M^*}{(4\pi\alpha)^{N/2}}\ \phi(x)\quad\mbox{uniformly in compact subsets of }\R^N.
\end{equation}
\end{coro}

\begin{rem}The limit~\eqref{eq:inner.limit} holds uniformly in sets of the form $\Omega\cap\{ |x|^2\le \delta(t)t\}$, as long as $\lim\limits_{t\to\infty}\delta(t)=0$.
\end{rem}

\end{document}